\documentclass[12pt,reqno]{amsart}
\usepackage{geometry}                
\usepackage{amsmath,amssymb}
\usepackage{graphicx}
\usepackage{amssymb,latexsym, amsmath}
\usepackage{amsfonts,amsthm}
\usepackage{amscd}
\usepackage{mathrsfs}
\usepackage{hyperref}
\usepackage{eucal}
\usepackage{epstopdf}
\usepackage{amsgen}
\usepackage{xspace}
\usepackage{verbatim}
\usepackage{stmaryrd}
\usepackage{enumitem}
\newlist{steps}{enumerate}{1}
\setlist[steps, 1]{label = Step \arabic*:}

\newcommand{\gd}{\Delta}

\newcommand{\inpt}[1]{\langle #1 \rangle}

\newcommand{\gw}{\Omega}

\newcommand{\gs}{\sigma}

\newcommand{\om}{\omega}

\newcommand{\nb}{\nabla}
\newcommand{\vp}{\varphi}
\newcommand{\ve}{\varepsilon}
\newcommand{\pdr}{\partial}

\newcommand{\tup}{\textup}

\newcommand{\beq}{\begin{equation}}
\newcommand{\eeq}{\end{equation}}
\newcommand{\bea}{\begin{align}}
\newcommand{\eea}{\end{align}}
\newcommand{\bthm}{\begin{theorem}}
\newcommand{\ethm}{\end{theorem}}
\newcommand{\bpr}{\begin{proof}}
\newcommand{\epr}{\end{proof}}
\newcommand{\bcl}{\begin{corollary}}
\newcommand{\ecl}{\end{corollary}}
\newcommand{\bpn}{\begin{proposition}}
\newcommand{\epn}{\end{proposition}}
\newcommand{\bre}{\begin{remark}}
\newcommand{\ere}{\end{remark}}
\newcommand{\bdf}{\begin{definition}}
\newcommand{\edf}{\end{definition}}
\newcommand{\bss}{\begin{align*}}
\newcommand{\ess}{\end{align*}}

\newcommand{\bl}{\label}

\newcommand{\zsw}{\theta_s \omega}
\newcommand{\ctw}{\theta_{-t} \omega}

\newtheorem{theorem}{Theorem}[section]
\newtheorem{corollary}[theorem]{Corollary}

\newtheorem{lemma}[theorem]{Lemma}
\newtheorem{proposition}[theorem]{Proposition}

\theoremstyle{definition}
\newtheorem{definition}[theorem]{Definition}
\theoremstyle{remark}
\newtheorem{remark}{Remark}

\numberwithin{equation}{section}

\begin{document}

\title[Stochastic Hindmarsh-Rose Equations]{Random Attractor for Stochastic Hindmarsh-Rose Equations with Multiplicative Noise}

\author[C. Phan]{Chi Phan}
\address{Department of Mathematics and Statistics, University of South Florida, Tampa, FL 33620, USA}
\email{chi@mail.usf.edu}
\thanks{}



\date{August 2, 2019}



\begin{abstract}
The longtime and global pullback dynamics of stochastic Hindmarsh-Rose equations  with multiplicative noise on a three-dimensional bounded domain in neurodynamics is investigated in this work. The existence of a random attractor for this random dynamical system is proved through the exponential transformation and uniform estimates showing the pullback absorbing property and the pullback asymptotically compactness of this cocycle in the $L^2$ Hilbert space.
\end{abstract}

\maketitle

\section{\textbf{Introduction}}. 

The Hindmarsh-Rose equations for neuronal spiking-bursting observed in experiments was initially proposed in \cite{HR1, HR2}. This mathematical model originally consists of three coupled nonlinear ordinary differential equations and has been studied through numerical simulations and mathematical analysis in recent years, cf. \cite{HR1, HR2, IG, MFL, SPH, Su} and the references therein. It exhibits rich bursting patterns, especially chaotic bursting and dynamics, as well as complex bifurcations. 

Very recently in \cite{PYS}, it is shown that there exists a global attractor for the diffusive and partly diffusive Hindmarsh-Rose equations in the deterministic environment.

In this work, we shall study the longtime random dynamics in terms of the existence of a random attractor for the stochastic diffusive Hindmarsh-Rose equations driven by a multiplicative white noise:
\begin{align}
    \frac{\pdr u}{\pdr t} & = d_1 \gd u +  \vp (u) + v - z + J + \ve u \circ \frac{dW}{dt}, \bl{suq} \\
    \frac{\pdr v}{\pdr t} & = d_2 \gd v + \psi (u) - v +  \ve v \circ \frac{dW}{dt}, \bl{svq} \\
    \frac{\pdr z}{\pdr t} & = d_3 \gd z + q (u - c) - rz +  \ve z \circ \frac{dW}{dt}, \bl{szq}
\end{align}
for $t > 0,\; x \in \gw \subset \mathbb{R}^{n}$ ($n \leq 3$), where $\gw$ is a bounded domain with locally Lipschitz continuous boundary, and the nonlinear terms 
\beq \bl{pp}
\vp (u) = au^2 - bu^3, \quad \text{and} \quad \psi (u) = \alpha - \beta u^2.
\eeq 
with the Neumann boundary condition
\begin{equation} \label{nbc1}
    \frac{\pdr u}{\pdr \nu} (t, x) = 0, \; \; \frac{\pdr v}{\pdr \nu} (t, x)= 0, \; \; \frac{\pdr z}{\pdr \nu} (t, x)= 0,\quad  t > 0,  \; x \in \partial \gw,
\end{equation}
and an initial condition
\begin{equation} \bl{inc1}
    u(0, x) = u_0 (x), \; v(0, x) = v_0 (x), \; z(0, x) = z_0 (x), \quad x \in \gw.
\end{equation}
Here $W(t), t \in \mathbb{R}$, is a one-dimensional standard Wiener process or called Brownian motion on the underlying probability space to be specified. The stochastic driving terms with the multiplicative noise indicate that the stochastic PDEs \eqref{suq}-\eqref{szq} are in the Stratonovich sense interpreted by the Stratonovich stochastic integrals and the corresponding differential calculus.

In this system \eqref{suq}-\eqref{szq}, the variable $u(t,x)$ refers to the membrane electric potential of a neuronal cell, the variable $v(t, x)$ represents the transport rate of the ions of sodium and potassium through the fast ion channels and is called the spiking variable, while the variables $z(t, x)$ represents the transport rate across the neuronal cell membrane through slow channels of calcium and other ions correlated to the bursting phenomenon and is called the bursting variable. 

Assume that all the parameters $a, b, \alpha, \beta, q, r, J$ and $\ve$ in the above equations are positive constants except $c \,(= u_R) \in \mathbb{R}$, which is a reference value of the membrane potential of a neuron cell. In the original model of ODE \cite{Su}, a set of the typical parameters are
\begin{gather*}
	J = 3.281, \;\; r = 0.0021, \;\; S = 4.0, \; \; q = rS,  \;\; c = -1.6,  \\[3pt]
	 \vp (s) = 3.0 s^2 - s^3, \;\; \psi (s) = 1.0 - 5.0 s^2.
\end{gather*}

\subsection{\textbf{The Hindmarsh-Rose Model in ODE}}

In 1982-1984, J.L. Hindmarsh and R.M. Rose developed the mathematical model to describe neuronal dynamics:
\begin{equation} \label{HR}
	\begin{split}
    \frac{du}{dt} & = au^2 - bu^3 + v - z + J,  \\
    \frac{dv}{dt} & = \alpha - \beta u^2  - v,  \\
    \frac{dz}{dt} & =  q (u - u_R) - rz.
    \end{split}
\end{equation}
This neuron model was motivated by the discovery of neuronal cells in the pond snail \emph{Lymnaea} which generated a burst after being depolarized by a short current pulse. This model characterizes the phenomena of synaptic bursting and especially chaotic bursting in a three-dimensional $(u, v, z)$ space. 

The chaotic dynamics is mainly reflected by the sensitive dependence of the longtime behavior of solutions on the initial conditions. The presence of the multiplicative noise as well as the diffusion of ions and membrane potential in the neuron model is expected to have large effect on the long-term behavior of the dynamical system in a random environment. 

The figure below is an illustration of the chaotic trajectories of the deterministic Hindmarsh-Rose model when the key parameter $J$ of the injected stimulation to the membrane potential varies.  
\begin{figure}[h]
		\includegraphics[width=10cm]{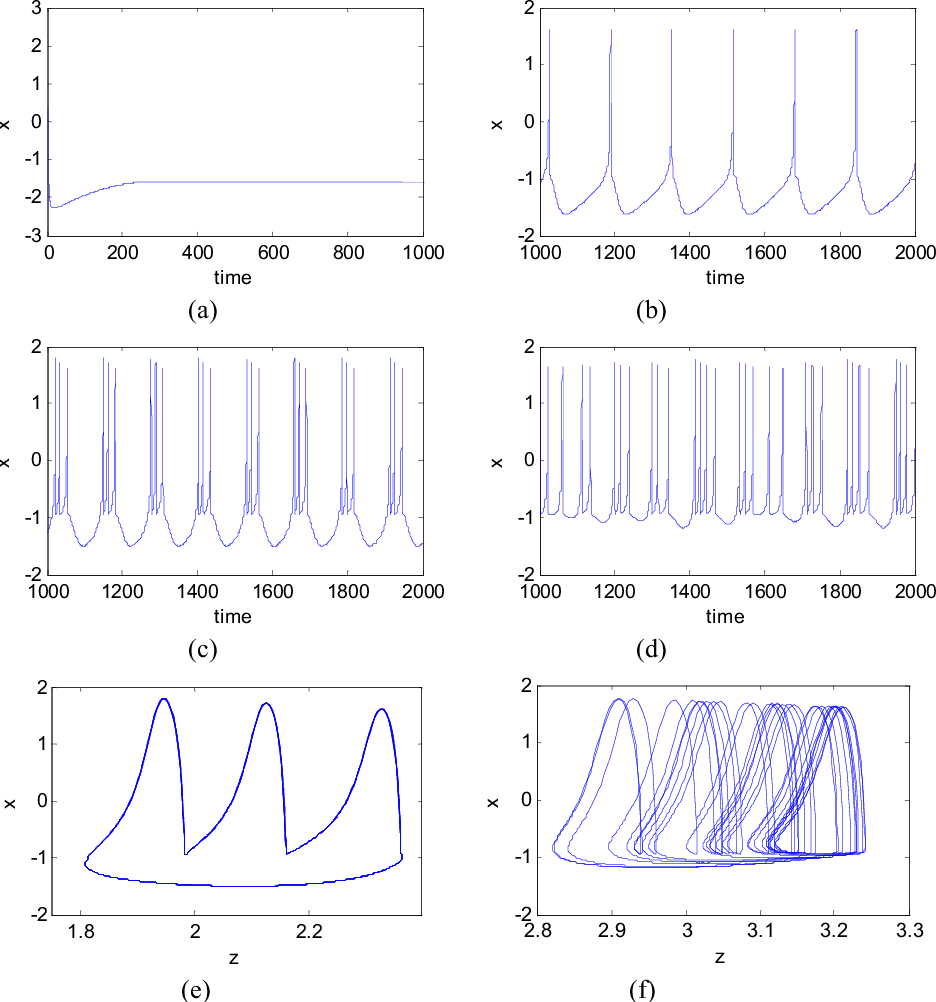}
		\caption{Time responses of the membrane potential for various value of the stimulated current: (a) resting state when J = 0, (b) tonic spiking when J = 1.2, (c) regular bursting when J = 2.2, (d) chaotic bursting when J = 3.1, (e) the x-z phase portrait when J = 2.2, (f) the x-z phase portrait when J = 3.1. Source: \cite{Ng}}
\end{figure}

Neuronal signals are short electrical pulses called spike or action potential. Neurons often exhibit bursts of alternating phases of rapid firing spikes and then quiescence. Bursting constitutes a mechanism to modulate and set the pace for brain functionalities and to communicate signals with the neighbor neurons. Bursting patterns occur in a variety of bio-systems such as pituitary melanotropic gland, thalamic neurons, respiratory pacemaker neurons, and insulin-secreting pancreatic $\beta$-cells, cf. \cite{BRS, CK,CS, HR2}. The current mathematical analysis of this neuron model mainly uses bifurcation theory together with numerical simulations, cf. \cite{BB, ET, EI, MFL, Ri, SPH, Tr, WS, Su}. 

Neurons communicate and coordinate actions through synaptic coupling or diffusive coupling (called gap junction) in neuroscience. Synaptic coupling of neurons has to reach certain threshold for release of quantal vesicles and form a synchronization \cite{DJ, Ru, SC}. 

The chaotic coupling exhibited in the simulations and analysis of this Hindmarsh-Rose model of ODE shows more rapid synchronization and more effective regularization of neurons due to \emph{lower threshold} than the synaptic coupling \cite{Tr, Su}. But the dynamics of chaotic bursting is highly complicated. 

It is known that Hodgkin-Huxley equations \cite{HH} (1952) provided a four-dimensional model for the dynamics of membrane potential taking into account of the sodium, potassium as well as leak ions current. The FitzHugh-Nagumo equations \cite{FH} (1961-1962) derived a two-dimensional model for an excitable neuron with the membrane potential and the current variable. This two-dimensional ODE model admits an exquisite phase plane analysis showing spikes excited by supra-threshold input pulses and sustained periodic spiking with refractory period, but due to the 2D nature FitzHugh-Nagumo equations exclude any chaotic solutions and chaotic dynamics so that no chaotic bursting can be generated.

The research on this model \eqref{HR} indicated the possibility to lower down the neuron firing threshold. More observations also indicate that the Hindmarsh-Rose model allows varying interspike-interval when the parameters vary. Therefore, the 3D model  \eqref{HR} is a suitable choice for the investigation of both the regular bursting and the chaotic bursting. It is expected that the augmented neuron model  of the stochastic Hindmarsh-Rose equations \eqref{suq}-\eqref{szq} studied in this paper will be exposed to a wide range of applications in neuroscience.

The rest of Section 1 is the formulation of the stochastic system \eqref{suq}-\eqref{szq} and provides basic concepts and results in the theory of random dynamics. In Section 2, we convert the stochastic PDEs to a system of random PDEs by the transformation of exponential multiplication. Then the global existence of pullback weak solutions is established. The uniform estimates will show the pullback absorbing property of the Hindmarsh-Rose semiflow in the $L^{2}$ space. In Section 3, we shall prove the main result on the existence of a random attractor for the diffusive Hindmarsh-Rose random dynamical system.

\subsection{\textbf{Preliminaries and Formulation}}

To study the stochastic dynamics in the asymptotically long run, we first recall the preliminary concepts for random dynamical systems, or called cocycles, cf. \cite{Ar, BLW, Ch, CDF, CF, EZ, FS, Ok, SH}. Let $(\mathfrak{Q}, \mathcal{F}, P)$ be a probability space and let $X$ be a real Banach space. 

\begin{definition} \label{sc1}
	$(\mathfrak{Q}, \mathcal{F}, P, \{{\theta_t}\}_{t \in \mathbb{R}})$ is called a $\textit{metric dynamical system}$ (MDS), if $(\mathfrak{Q}, \mathcal{F}, P)$ is a probability space and $\theta_t$ is a time-shifting mapping with the following conditions satisfied:
	
	(i) the mapping $\theta : \mathbb{R} \times \mathfrak{Q} \to \mathfrak{Q}$ is $(\mathscr{B} (\mathbb{R}) \otimes \mathcal{F}, \mathcal{F})$ - measurable,
	
	(ii) $\theta_0$ is the identity on $\mathfrak{Q}$,
	
	(iii) $\theta_{t + s} = \theta_t \circ \theta_s$ for all $t, s \in \mathbb{R}$, and
	
	(iv) $\theta_t$ is probability invariant, meaning $\theta_t P = P$ for all $t \in \mathbb{R}$.
	
\noindent Here $\mathscr{B} (X)$ stands for the $\sigma$-algebra of Borel sets in a Banach space $X$ and $(\theta_t P)(S) = P(\theta_t S)$ for any $S \in \mathcal{F}$.
\end{definition}

\begin{definition} \label{sc2}
	A continuous $\textit{random dynamical system}$ (RDS) briefly called a cocycle on $X$ over an MDS $(\mathfrak{Q}, \mathcal{F}, P, \{{\theta_t}\}_{t \in \mathbb{R}})$ is a mapping
	$$
	\varphi (t, \omega, x) : [0, \infty) \times  \mathfrak{Q} \times X \to X,
	$$
	which is $(\mathscr{B} (\mathbb{R^+}) \otimes \mathcal{F} \otimes \mathscr{B} (X), \mathscr{B} (X))$- measurable and satisfies the following conditions for every $\omega$ in $\mathfrak{Q}$:
	
	(i) $\varphi (0, \omega, \cdot)$ is the identity operator on $X$.
	
	(ii) The cocycle property holds:
	$$
	\varphi (t + s, \omega, \cdot) = \varphi (t, \theta_s \omega, \varphi (s, \omega, \cdot)), \quad \text{for all} \; t, s \geq 0.
	$$
	
	(iii) The mapping $\varphi (\cdot, \omega, \cdot) : [0, \infty) \times X \to X$ is strongly continuous.
\end{definition}

\begin{definition} \label{sc3}
	A set-valued function $B : \mathfrak{Q} \to 2^X$ is a random set in $X$ if its graph $\{(\omega, x) : x \in B(\omega)\} \subset \mathfrak{Q} \times X$ is an element of the product $\sigma$-algebra $\mathcal{F} \otimes \mathscr{B} (X)$. A $\textit{bounded}$ random set $B(\omega) \subset X$ means that there is a random variable $r(\omega) \in [0, \infty), \omega \in \mathfrak{Q}$, such that $\interleave {B(\omega)} \interleave := \sup_{x \in B(\omega)} \|x\| \leq r(\omega)$ for all $\omega \in \mathfrak{Q}$. A bounded random set $B(\om)$ is called $\textit{tempered}$ with respect to $\{\theta_t\}_{t \in \mathbb{R}}$ on $(\mathfrak{Q}, \mathcal{F}, P)$, if for any $\omega \in \mathfrak{Q}$ and for any constant $\beta > 0$,
		$$
		\lim_{t \to \infty} e^{-\beta t} \interleave B(\theta_{-t} \omega)\interleave = 0.
		$$
	A random set $S(\omega) \subset X$ is called $\textit{compact}$ (reps. $\textit{precompact}$) if for every $\omega \in \mathfrak{Q}$ the set $S(\omega)$ is a compact (reps. precompact) set in $X$.
\end{definition}

\begin{definition} \label{sc4}
	A random variable $R : (\mathfrak{Q}, \mathcal{F}, P) \to (0, \infty)$ is called $ \textit{tempered with}$ $\textit{respect to a metric dynamical system}$ $\{\theta_t\}_{t \in \mathbb{R}}$ on $(\mathfrak{Q}, \mathcal{F}, P)$, if for any $\om \in \mathfrak{Q}$,
	$$
	\lim_{t \to -\infty} \frac{1}{t} \; \text{log} \; R(\theta_t \omega) = 0.
	$$
\end{definition}

\begin{remark}
	If $\{B(\omega)\}_{\omega \in \mathfrak{Q}}$ is a closed random set of $X$ such that for any fixed $x \in X$ the mapping $\omega \mapsto d(x, B(\omega)) = \inf \{ \|x - y\| : y \in B(\omega)\}$ is $(\mathcal{F}, \mathscr{B} (\mathbb{R^+})$-measurable, then $B(\omega)$ is a random set in the sense of Definition \ref{sc3}, cf. \cite{Castaing}. Considering that a random set may be neither closed or open, Definition \ref{sc3} is more general.
\end{remark}

We shall let $\mathscr{D}_X$ denote an inclusion-closed family of random sets in $X$, meaning that if $D = \{D(\omega)\}_{\omega \in \mathfrak{Q}} \in \mathscr{D}_X$ and $\hat{D} = \{\hat{D}(\omega)\}_{\omega \in \mathfrak{Q}}$ with $\hat{D}(\omega) \subset D(\omega)$ for all $\omega \in \mathfrak{Q}$, then $\hat{D} \in \mathscr{D}_X$. Such a family of random sets in $X$ is called a $\textit{universe}$. In this work, we define $\mathscr{D}_H$ to be the universe of all the tempered random sets in the Hilbert space $H = L^2 (\gw, \mathbb{R}^3)$.

\begin{definition} \label{sc5}
	For a given universe $\mathscr{D}_X$ of random sets in a Banach space $X$, a random set $K \in \mathscr{D}_X$ is called a $\textit{pullback absorbing set}$ with respect to an RDS (cocycle) $\varphi$ over the MDS $(\mathfrak{Q}, \mathcal{F}, P, \{{\theta_t}\}_{t \in \mathbb{R}})$, if for any bounded random set $B \in \mathscr{D}_X$ and any $\omega \in \mathfrak{Q}$ there exists a finite time $T_B (\omega) > 0$ such that
	$$
	\varphi (t, \theta_{-t} \omega, B (\theta_{-t} \omega)) \subset K (\omega), \quad \text{for all} \;\;  t \geq T_B (\omega).
	$$
\end{definition}

\begin{definition} \label{sc6}
	Let a universe $\mathscr{D}_X$ of random sets in a Banach space $X$ be given, A random dynamical system (cocycle) $\varphi$ is $\textit{pullback asymptotically compact}$ with respect to $\mathscr{D}_X$ , if for any $\omega \in \mathfrak{Q}$, the sequence
	$$
	\{\varphi (t_m, \theta_{- t_m} \omega, x_m)\}^\infty_{m = 1} \; \text{has a convergent subsequence in} \; X,
	$$
	whenever $t_m \to \infty$ and $x_m \in B (\theta_{-t} \omega)$ for any given $B \in \mathscr{D}_X$.
\end{definition}

\begin{definition} \label{sc7}
	Let a universe $\mathscr{D}_X$ of tempered random sets in a Banach space $X$ be given. A random set $\mathcal{A} \in \mathscr{D}_X$ is called a $\textit{random attractor}$ for a given random dynamical system (cocycle) $\varphi$ over the metric dynamical system $(\mathfrak{Q}, \mathcal{F}, P, \{{\theta_t}\}_{t \in \mathbb{R}})$, if the following conditions are satisfied:
	
	(i) $\mathcal{A}$ is a compact random set in the space $X$.
	
	(ii) $\mathcal{A}$ is invariant in the sense that 
	$$
	\varphi (t, \omega, \mathcal{A}(\omega)) = \mathcal{A}(\theta_t \omega), \quad \text{for all} \;\; t \geq 0, \; \om \in \mathfrak{Q}.
	$$
	
	(iii) $\mathcal{A}$  attracts every $B \in \mathscr{D}_X$ in the pullback sense that 
	$$
	\lim_{t \to \infty} dist_X (\varphi (t, \theta_{-t} \omega, B (\theta_{-t} \omega)), \mathcal{A} (\omega)) = 0, \quad \om \in \mathfrak{Q},
	$$
	where $dist_X (\cdot, \cdot)$ is the Hausdorff semi-distance with respect to the $X$-norm. Then $\mathscr{D}_X$ is called the $\textit{basin}$ of attraction for $\mathcal{A}$.
\end{definition}	

The existence of random attractors for continuous and discrete random dynamical systems has been studied in the recent three decades by many authors, cf. \cite{Ar, BLW, CGA, Ch, CDF, CF, HSZ, SH, Sm, SWW, W12, W14, W19, Y14, YY14, Y17, Zs}.  The following theorem is shown in \cite{CF, SH}.

\begin{theorem} \bl{sc8}
	Given a Banach space $X$ and a universe $\mathscr{D}_X$ of random sets in $X$, let $\varphi$ be a continuous random dynamical system on $X$ over the metric dynamical system $(\mathfrak{Q}, \mathcal{F}, P, \{{\theta_t}\}_{t \in \mathbb{R}})$. If the following two conditions are satisfied:
	
	\textup{(i)} there exists a closed pullback absorbing set $K = \{K (\omega)\}_{\omega \in \mathfrak{Q}} \in \mathscr{D}_X$ for $\varphi$,
	
	\textup{(ii)} the cocycle $\varphi$ is pullback asymptotically compact with respect to $\mathscr{D}_X$, 
	
\noindent then there exists a unique random attractor $\mathcal{A} = \{\mathcal{A} (\omega)\}_{ \omega \in \mathfrak{Q}} \in \mathscr{D}_X$ for the cocycle $\varphi$ and the random attractor is given by
	$$
	\mathcal{A} (\omega) = \bigcap_{\tau \geq 0} \; {\overline{\bigcup_{ t \geq \tau} \varphi (t, \theta_{-t} \omega, K (\theta_{-t} \omega))}}, \quad  \om \in \mathfrak{Q}.
	$$
\end{theorem}

We now formulate the initial-boundary value problem \eqref{suq}--\eqref{inc1} of the stochastic Hindmarsh-Rose equations with the multiplicative white noise in the framework of the product Hilbert spaces
\beq \bl{srd1}
H = L^2 (\gw, \mathbb{R}^3) \quad \text{and} \quad E = H^1 (\gw, \mathbb{R}^3).
\eeq
The norm and inner-product of $H$ or $L^2 (\gw)$ will be denoted by $\| \cdot\|$ and $\inpt{\cdot, \cdot}$, respectively. The norm of space $E$ will be denoted by $\| \cdot\|_E$. The norm of $L^p (\gw) $ or $L^p (\gw, \mathbb{R}^3)$ will be denoted by $\| \cdot\|_{L^p}$ for $p \neq 2$. W use $| \cdot|$ to denote a vector norm in Euclidean spaces. 

The nonpositive self-adjoint linear differential operator
\begin{equation} \label{srd2}
A =
\begin{pmatrix}
d_1 \gd  & 0   & 0 \\[3pt]
0 & d_2 \gd  & 0 \\[3pt]
0 & 0 & d_3 \gd
\end{pmatrix}
: D(A) \rightarrow H,
\end{equation}
where 
\begin{equation*}
D(A) = \left\{(\varphi, \phi, \zeta) \in H^2 (\gw, \mathbb{R}^3) : \frac{\partial \vp}{\partial \nu} = \frac{\partial \phi}{\partial \nu} = \frac{\partial \zeta}{\partial \nu} = 0 \;\;\text{on} \;\; \partial \gw \right\}
\end{equation*}
is the generator of an analytic $C_0$-semigroup $\{e^{At}\}_{t \geq 0}$ of contraction on the Hilbert space $H$. By the Sobolev embedding $H^{1}(\gw) \hookrightarrow L^6(\gw)$ for space dimension $n \leq 3$, the nonlinear mapping 
\begin{equation} \label{srd3}
f(u,v, z) =
\begin{pmatrix}
\vp (u) + v - z + J \\[4pt]
\psi (u) - v,  \\[4pt]
q (u - c) - rz
\end{pmatrix}
: E \longrightarrow H
\end{equation}
is locally Lipschitz continuous. Thus the initial-boundary value problem \eqref{suq}--\eqref{inc1} is formulated into an initial value problem of the following stochastic Hindmarsh-Rose evolutionary equation driven by a multiplicative white noise,
\begin{equation} \label{srd4}
\begin{split}
\frac{dg}{dt} &= A g + f(g) + \ve g \circ \frac{dW}{dt}, \quad t > \tau \in \mathbb{R}, \; \om \in \mathfrak{Q}, \\
&g (\tau) = g_0 = (u_0, v_0, z_0) \in H.
\end{split}
\end{equation}
Here $g(t, \omega, g_0) = \text{col} \, (u(t, \cdot, \omega, g_0), v(t, \cdot, \omega, g_0), z(t, \cdot, \omega, g_0))$, where dot stands for the hidden spatial variable $x$.

Assume that $\{W(t)\}_{t \in \mathbb{R}}$ is a one-dimensional, two-sided standard Wiener process in the probability space $(\mathfrak{Q}, \mathcal{F}, P)$, where the sample space
\beq \bl{scrd5}
\mathfrak{Q} = \{\omega \in C(\mathbb{R}, \mathbb{R}) : \omega (0) = 0\}
\eeq
where $C(\mathbb{R}, \mathbb{R})$ stands for the metric space of continuous functions on the real line, the $\sigma$-algebra $\mathcal{F}$ is generated by the compact-open topology endowed in $\mathfrak{Q}$, and $P$ is the corresponding Wiener measure \cite{Ar, Ch, Ok} on $\mathcal{F}$. Define the $P$-preserving time-shift transformations $\{\theta_t\}_{t \in \mathbb{R}}$ by
\beq \bl{srd6}
     (\theta_t \omega) (\cdot) = \omega (\cdot + t) - \omega (t), \quad \text{for} \;\; t \in \mathbb{R},\; \omega \in \mathfrak{Q}.
\eeq
Then $(\mathfrak{Q}, \mathcal{F}, P, \{{\theta_t}\}_{t \in \mathbb{R}})$ is a metric dynamical system and the stochastic process $\{W (t, \omega) = \omega (t) : t \in \mathbb{R},\, \omega \in \mathfrak{Q}\}$ is the canonical Wiener process. Accordingly $dW/dt$ in \eqref{srd4} denotes the white noise. The results we shall prove in this paper can be extended to a vector white noise with three different but independent scalar noises in the three component equations.

In the recent paper \cite{PYS}, we have shown the existence of a global attractor for the diffusive deterministic Hindmarsh-Rose equations and for the partly diffusive Hindmarsh-Rose equations in the space $H$. In this paper, it will be shown that there exists a random attractor in the space $H$ for the random dynamical system generated by the global solutions of the stochastic evolutionary equation \eqref{srd4}. 

\section{\textbf{Random Hindmarsh-Rose Equations and Pullback Dissipativity}}

The mathematical treatment of the stochastic PDE such as in the form of \eqref{suq}-\eqref{szq} driven by the multiplicative noise will be facilitated by its conversion to a random PDE with coefficients and initial data being random variables instead. For this purpose, one can exploit the following properties of the Wiener process. 

\begin{proposition} \label{sc9}
	Let the MDS $(\mathfrak{Q}, \mathcal{F}, P, \{{\theta_t}\}_{t \in \mathbb{R}})$ and the Wiener process $W(t)$ be defined as above. Then the following statements hold.
	
	\textup{(1)} The Wiener process $W(t)$ has the asymptotically sublinear growth property,
	\beq \bl{srd7}
	\lim_{t \to \pm \infty} \frac{|W(t)|}{|t|} = 0, \quad \text{a.s.}
	\eeq
	
	\textup{(2)} For any given positive constant $\lambda$, the stochastic process $X(t) = e^{-\lambda W(t)}$ is a solution of the following stochastic differential equation in the Stratonovich sense, 
	\beq \bl{srd8}
	dX_t = - \lambda X_t \circ d W_t.
	\eeq
	
	\textup{(3)} $W(t)$ is locally H\"older continuous with exponents $\gamma \in \left(0, \frac{1}{2}\right)$. It means that for any integer n,
	\beq \bl{srd9}
	\sup_{n \leq s < t \leq n + 1} \frac{|W(t) - W(s)|}{|t -s|^\gamma} < \infty, \quad \text{a.s.}
	\eeq
\end{proposition}	

\begin{proof}
	By the law of iterated logarithm \cite{Ok},
	$$
	\lim_{t \to \pm \infty} \sup \frac{|W(t)|}{\sqrt{2 |t|\; \text{log log}\;|t|}} =1, \quad \text{a.s.}
	$$
	Then \eqref{srd7} is valid. Next, from It\^o's formula \cite{Ok} we have
	$$
	dX_t = - \lambda e^{- \lambda W_t} dW_t + \frac{1}{2} \, \lambda^2 e^{- \lambda W_t} dt.
	$$
	On the other hand, the transformation formula \cite{Ok} of the stochastic It\^{o} integral and the Stratonovich integral  reads 
	$$
	h(W_t) \circ dW_t \;(\text{Stratonovich sense}) = h(W_t) dW_t \;(\text{It\^o sense}) + \frac{1}{2}\, h' (W_t) dt, 
	$$
	as long as $h(W_t)$ and $h'(W_t)$ are locally $L^2$-integrable. Set $h(\omega) = \lambda e^{- \lambda \omega}$ in the above equality. Then 
	$$
	- \lambda X_t \circ dW_t = - \lambda e^{- \lambda W_t} dW_t + \frac{1}{2}\, \lambda^2 e^{- \lambda W_t} dt.
	$$
	Hence \eqref{srd8} holds. Finally, \eqref{srd9} follows from the Kolmogorov Moment Criterion.
\end{proof}

       We now convert the stochastic PDE \eqref{suq} - \eqref{szq} to a system of random PDE by the exponential multiplication of $Q(t, \om) = e^{- \ve \om (t)}$:
\beq \bl{etrans}
	U(t) = Q(t, \om) u(t), \quad V(t) = Q(t, \om) v(t), \quad Z(t) = Q(t, \om) z(t).
\eeq
According to the second statement in Proposition \ref{sc9}, the initial-boundary value problem \eqref{suq}--\eqref{inc1} is equivalently converted to the following system of random PDEs:
\begin{align}
    \frac{\pdr U}{\pdr t} & = d_1 \gd U +  \frac{a}{Q(t, \om)} U^2 - \frac{b}{Q(t, \om)^2} U^3 + V - Z + J Q(t, \om), \bl{sUq} \\
    \frac{\pdr V}{\pdr t} & = d_2 \gd V + \alpha Q(t, \om) - \frac{\beta}{Q(t, \om)} U^2 - V, \bl{sVq} \\
    \frac{\pdr Z}{\pdr t} & = d_3 \gd Z + q (U - c \,Q(t, \om)) - r Z, \bl{sZq}
\end{align}
for $\om \in \mathfrak{Q}, \, t > 0,\, x \in \gw \subset \mathbb{R}^{n}$ ($n \leq 3$), with the boundary condition
\begin{equation} \label{nbc2}
    \frac{\pdr U}{\pdr \nu} (t, x, \om) = 0, \; \frac{\pdr V}{\pdr \nu} (t, x, \om)= 0, \; \frac{\pdr Z}{\pdr \nu} (t, x, \om)= 0, \quad t \geq \tau \in \mathbb{R}, \;  x \in \partial \gw,
\end{equation}
and an initial condition for $\om \in \mathfrak{Q}$,
\begin{equation} \bl{inc2}
    (U(\tau, x, \om), V(\tau, x, \om), Z(\tau, x, \om)) = Q(\tau, \om)(u_0 (x), v_0 (x), z_0 (x)), \; \; x \in \gw.
\end{equation}
The equations \eqref{sUq}-\eqref{sZq} are pathwise nonautonomous random PDEs and \eqref{sUq}-\eqref{inc2} can be written as the initial value problem of the random evolutionary equation:
\beq \bl{sGq}
	\begin{split}
	&\frac{\pdr G}{\pdr t}  =  AG  +  F(G, \theta_t \om), \quad t \geq \tau \in \mathbb{R}, \; \om \in \mathfrak{Q}, \\
	& G(\tau, \om) = G_\tau (\om) = Q(\tau, \om) (u_0, v_0, z_0), \; \om \in \mathfrak{Q},
	\end{split}
\eeq
for any $g_0 = (u_0, v_0, z_0) \in H$. Here we define the weak solution of the initial value problem \eqref{sGq} with the initial state $G_{\tau} = Q(\tau, \om)g_0$,
$$
	G(t, \om; \tau, G_\tau) = Q(t, \om) \begin{pmatrix}
	u \\
	v \\
	z
	\end{pmatrix} 
	(t, \cdot, \,\om; \, \tau, G_\tau) = \begin{pmatrix}
	U \\
	V \\
	Z
	\end{pmatrix} 
	(t, \cdot, \,\om; \, \tau, G_\tau),
$$
to be the pathwise weak solution \cite[page 283]{CV} of the nonautonomous initial-boundary problem \eqref{sUq}-\eqref{inc2}, specified as in \cite[Definition 2.1]{Y12}.

By conducting \emph{a priori} estimates on the Galerkin approximate solutions of the equations \eqref{sUq}-\eqref{sZq} and the compactness argument outlined in \cite[Chapter II and XV]{CV} with some adaptations, we can prove the local existence and uniqueness of the weal solution $G(t, \om)$ in the space $H$ on a local time interval $t \in [\tau, T(\om, G_\tau)]$, and the solution is continuously depending on the initial data. Further by the parabolic regularity \cite[Theorem 48.5]{SY}, every weak solution becomes a strong solution in the space $E$ when $t > \tau$ in the time interval of existence. Every weak solution $G(t, \om)$ of the problem \eqref{sGq} on the maximal existence interval has the property 
\beq 
	G \in C([\tau, T_{max}), H) \cap C^1 ((\tau, T_{max}), H) \cap L^2_{loc} ([\tau, T_{max}), E).
\eeq

\subsection{\textbf{Global Existence of Pullback Solutions}}

In this section, we first prove the global existence of all the pullback weak solutions of the problem \eqref{sUq}-\eqref{inc2} and to explore the dissipativity of the generated random dynamical system.

\begin{lemma} \label{theo1}
	There exists a random variable $r_0 (\om) > 0$ depending only on the parameters such that, for any given random variable $\rho(\om) > 0$, there is a time $-\infty < \tau(\rho,\om) \leq -1$ and the following statement holds. For any  $t_0 \leq \tau(\rho,\om)$ and for any initial data $g_0 = (u_0, v_0, z_0) \in H$ with $\|g_0\| \leq \rho(\om)$, the weak solution $G(t,\om)$ of the problem \eqref{sGq} with $G(t_0,\om) = Q(t_0,\om) g_0$ uniquely exists on $[t_0,-1]$ and satisfies 
	\beq \bl{srd10}
	\|G(-1,\,\om;\,t_0,\,Q(t_0,\om)g_0)\| \leq r_0(\om),\quad \;\; \om \in \mathfrak{Q}.
	\eeq
\end{lemma} 
\begin{proof}
	Take the $L^2 (\gw)$ inner-product $\langle \eqref{sUq}, c_1 U) \rangle$, $\langle \eqref{sVq}, V) \rangle$ and $\langle \eqref{sZq}, Z) \rangle$ with constant $c_1 > 0$ to be determined later, we obtain the following:
	\beq \bl{srd11}
	\begin{split}
		&\frac{1}{2}\frac{d}{dt} \left(c_1 \|U\|^2+ \|V\|^2 + \|Z\|^2\right) + \left(c_1 d_1 \|\nb U\|^2 + d_2\|\nb V\|^2 + d_3\|\nb Z\|^2\right) \\
		= &\, \int_\gw c_1 \left(\frac{a}{Q(t,\om)}U^3 - \frac{b}{Q(t,\om)^2}U^4 +UV - UZ + JUQ(t,\om)\right)dx\\
		+ &\, \int_\gw \left(\alpha V Q(t,\om) - \frac{\beta }{Q(t,\om)}U^2 V - V^2 + q (U - c Q(t,\om))Z - rZ^2\right) dx \\
		\leq &\,  \int_\gw c_1 \left(\frac{a}{Q(t,\om)} U^3- \frac{b}{Q(t,\om)^2}U^4 +UV - UZ + JUQ(t,\om)\right)dx\\
		+ &\, \int_\gw \left\{\left(2\alpha^2 Q(t,\om)^2 + \frac{\beta^2}{2 Q(t,\om)^2} U^4 - \frac{3}{8}V^2 \right) + \left[\frac{q^2}{r}(U^2 + c^2 Q(t,\om)^2) - \frac{1}{2}rZ^2\right]\right\} dx.
	\end{split}
	\eeq
	Choose the positive constant in \eqref{srd11} to be $c_1 = \frac{1}{b}(\beta^2 + 3)$ so that 
	$$
	- c_1\int_\gw \frac{b}{Q(t,\om)^2}U^4\, dx + \int_\gw \frac{\beta^2}{Q(t,\om)^2} U^4\, dx \leq -3 \int_\gw \frac{U^4}{Q(t,\om)^2}\, dx.
	$$
	By Young's inequality, we have
	\begin{equation*}
	\begin{split}
		\int_\gw \frac{c_1 a}{Q(t,\om)} U^3\,dx &\,\leq \frac{3}{4}\int_\gw \frac{U^4}{Q(t,\om)^2}\,dx + \frac{1}{4} (c_1 a Q(t,\om))^4 |\gw | \\[3pt]
		&\,\leq \int_\gw \frac{U^4}{Q(t,\om)^2} \,dx + \left(c_1 a\, Q(t,\om)\right)^4 | \gw |,
	\end{split}
	\end{equation*}
	as well as
	\begin{equation} \bl{linU}
	\begin{split}
		\int_\gw c_1 (UV - UZ + JUQ(t,\om))\,dx &\,\leq \int_\gw \left[2(c_1 U)^2 + \frac{1}{8}V^2 + \frac{(c_1 U)^2}{r}+\frac{1}{4} rZ^2 \right. \\
		&\quad \quad \quad \left. +\frac{1}{2}(c_1 U)^2 + \frac{1}{2}J^2 Q(t,\om)^2\right]\,dx. 
	\end{split}
	\end{equation}
	Collecting those integral terms of $U^2$ on the right-hand side in \eqref{srd11} and in \eqref{linU}, we obtain
	\begin{gather*}
	\int_\gw \left[2(c_1U)^2 + \frac{(c_1 U)^2}{r}+ \frac{1}{2}(c_1 U)^2 + \frac{q^2}{r}U^2 \right]dx \\
	\quad \leq \int_\gw \frac{U^4}{Q(t,\om)^2} \,dx + \left[c_1^2 \left(\frac{5}{2} + \frac{1}{r} \right) + \frac{q^2}{r}\right]^2 Q(t,\om)^2 | \gw |.
	\end{gather*}
	Substitute the above inequalities with respect to the integral terms of $U^4, U^3$ and $U^2$ into \eqref{srd11}. Then we get
	\beq \bl{srd12}
	\begin{split}
		&\frac{1}{2}\frac{d}{dt} \left(c_1 \|U\|^2+ \|V\|^2 + \|Z\|^2\right) + \left(c_1 d_1 \|\nb U\|^2 + d_2\|\nb V\|^2 + d_3\|\nb Z\|^2\right) \\[3pt]
		\leq &\, \int_\gw \left[\frac{2- 3}{Q(t,\om)^2}U^4 + \left(\frac{1}{8} - \frac{3}{8}\right) V^2 +  \left( \frac{1}{4} - \frac{1}{2}\right) rZ^2\right]\,dx \\
		+ &\,\left[\frac{1}{2}J^2 + \left(c_1^2 \left(\frac{5}{2} + \frac{1}{r}\right) + \frac{q^2}{r} \right)^2 + 2\alpha^2 + \frac{q^2c^2}{r}\right] Q(t,\om)^2 |\gw | + (c_1 a)^4 Q(t,\om)^4 | \gw | \\
		\leq &\, - \int_\gw \left(\frac{1}{Q(t,\om)^2}U^4 (t,x) + \frac{1}{4} V^2 (t,x) + \frac{1}{4} rZ^2(t,x)\right) dx. \\[5pt]
		&\, + (c_1 a)^4 Q(t,\om)^4 | \gw | + c_2 \,Q(t,\om)^2 |\gw |,
	\end{split}
	\eeq
	where
	$$
	c_2 = \frac{1}{2}J^2 + \left[c_1^2 \left(\frac{5}{2} + \frac{1}{r}\right) + \frac{q^2}{r} \right]^2 + 2\alpha^2 + \frac{q^2c^2}{r}.
	$$
	Let $d = \min \{d_1, d_2, d_3\}$. Then the inequality \eqref{srd12} implies
	\begin{equation*}
	\begin{split}
	&\, \frac{d}{dt}  (c_1 \|U(t)\|^2 + \|V(t)\|^2 +\|Z(t)\|^2) + 2d (c_1 \|\nb u\|^2 + \|\nb v\|^2 + \|\nb w\|^2) \\
	&\, + \int_\gw \left(\frac{2}{Q(t,\om)^2}U^4 (t, x) + \frac{1}{2}V^2 (t,x) + \frac{1}{2}rZ^2(t,x)\right) dx \\[3pt]
	\leq &\; 2 c_2\, Q(t,\om)^2 | \gw | + 2 (c_1 a)^4 Q(t,\om)^4 \gw |.
	\end{split}
	\end{equation*}
	Moreover, we have
	$$
	\frac{2}{Q(t,\om)^2}U^4 \geq \frac{1}{2}\left(c_1 U^2 - \frac{c_1^2 Q(t,\om)^2}{16}\right).
	$$
	Therefore, 
	\begin{equation} \bl{srd13}
	\begin{split}
	&\, \frac{d}{dt} (c_1 \|U(t)\|^2 + \|V(t)\|^2 +\|Z(t)\|^2) + 2d (c_1 \|\nb U\|^2 + \|\nb V\|^2 + \|\nb Z\|^2) \\[3pt]
	&\, + \frac{1}{2} (c_1 \|U(t)\|^2 + \|V(t)\|^2 + r\|Z(t)\|^2)  \\
	\leq &\, \left(2 c_2 + \frac{1}{32}c_1^2\right)Q(t,\om)^2 | \gw | + 2 (c_1 a)^4 Q(t,\om)^4 |\gw |,
	\end{split}
	\end{equation}
	for $t \in [\tau, T_{max})$. Set $\sigma = \frac{1}{2} \min \{1, r\}$. Then the Gronwall inequality is applied to the reduced inequality \eqref{srd13}	,
	\begin{equation*}
	\begin{split}
	&\frac{d}{dt} (c_1 \|U(t)\|^2 + \|V(t)\|^2 +\|Z(t)\|^2) + \sigma (c_1 \|U(t)\|^2 + \|V(t)\|^2 + \|Z(t)\|^2) \\
	\leq &\, \left(2 c_2 + \frac{1}{32}c_1^2\right)Q(t,\om)^2 | \gw | + 2 (c_1 a)^4 Q(t,\om)^4 | \gw | 
	\end{split}
	\end{equation*}
	and shows that
	\beq \bl{srd14}
	\begin{split}
		&c_1 \|U(t)\|^2 + \|V(t)\|^2 +\|Z(t)\|^2 \leq e^{- \sigma (t - t_0)}(c_1 \|U_0\|^2 + \|V_0\|^2 +\|Z_0\|^2) \\
		+ &\,  \int_{-\infty}^t e^{- \sigma(t-s)} \left[\left(2 c_2 + \frac{1}{32}c_1^2\right)Q(s,\om)^2 | \gw | + 2 (c_1 a)^4 Q(s,\om)^4 | \gw | \right]\, ds, \; \;t \in [\tau, T_{max}).
	\end{split}
	\eeq
	We obtain
	\beq \bl{srd15}
	\begin{split}
			&\|U(t)\|^2 + \|V(t)\|^2 +\|Z(t)\|^2 \leq \frac{\text{max}\{c_1, 1\}}{\text{min}\{c_1, 1\}}\, e^{- \sigma (t - t_0)}(\|U_0\|^2 + \|V_0\|^2 +\|Z_0\|^2) \\
			+ &\, \frac{| \gw |}{\text{min}\{c_1, 1\}} \int_{-\infty}^t e^{- \sigma(t-s)} \left[\left(2 c_2 + \frac{1}{32}c_1^2\right)Q(s,\om)^2 + 2 (c_1 a)^4 Q(s,\om)^4 \right]\, ds.
	\end{split}
	\eeq
	Hence, the solutions of the initial value problem of the equation \eqref{sGq} satisfies the bounded estimate
	\beq \bl{srd16}
	\begin{split}
		&\|G(t,\om;t_0,Q(t_0,\om)g_0)\|^2 \leq \frac{\|Q(t_0,\om)\|^2 \text{max}\{c_1, 1\}}{\text{min}\{c_1, 1\}}\, e^{- \sigma (t - t_0)}\, \|g_0\|^2\\
	    &\; + \frac{| \gw |}{\text{min}\{c_1, 1\}} \int_{-\infty}^t e^{- \sigma(t-s)} \left[\left(2 c_2 + \frac{1}{32}c_1^2\right)Q(s,\om)^2 + 2 (c_1 a)^4 Q(s,\om)^4 \right]\, ds,\quad t \geq t_0.
	\end{split}
	\eeq
	Take $t = -1$ and substitute $Q(t,\om) = e^{- \ve \om (t)}$ into \eqref{srd16}. We then get
	\beq \bl{srd17}
	\begin{split}
		&\|G(-1,\om;t_0,Q(t_0,\om)g_0)\|^2 \leq \frac{ \text{max}\{c_1, 1\}}{\text{min}\{c_1, 1\}} e^{\sigma - \sigma |t_0| - 2\ve\, \om (t_0)} \|g_0\|^2 \\
		&\; + \frac{| \gw |}{\text{min}\{c_1, 1\}} \int_{-\infty}^{-1} e^{\sigma + \sigma s} \left[\left(2 c_2 + \frac{1}{32}c_1^2\right)e^{- 2 \ve \om (s)} + 2 (c_1 a)^4 e^{- 4 \ve \om (s)} \right] ds.
	\end{split}
	\eeq
	Note that
	$$
	e^{- \sigma |t_0| - 2\ve\om (t_0)} = \text{exp}\left(-\sigma |t_0|\left[1 + \frac{2\ve\, \om (t_0)}{\sigma |t_0|}\right]\right) = \text{exp}\left(-\sigma |t_0|\left[1 - \frac{2\ve \,\om (t_0)}{\sigma t_0}\right]\right). 
	$$
	By the asymptotically sublinear property \eqref{srd7}, for any given random variable $\rho(\om) > 0$ and for a.e. $\om \in \mathfrak{Q}$, there exist a time $\tau(\rho,\om) \leq -1$ such that for any $t_0 \leq \tau(\rho,\om)$, we have 
	\beq \bl{srd18}
	1 - \frac{2\ve\om (t_0)}{\sigma t_0} \geq \frac{1}{2}\quad \text{and}\quad e^{\sigma (1 - \frac{1}{2}|t_0|)}\frac{ \text{max}\{c_1, 1\}}{\text{min}\{c_1, 1\}}\rho^2(\om) \leq 1.
	\eeq
	Therefore, from \eqref{srd17}, we obtain
	\beq \bl{srd19}
	\|G(-1,\om;t_0,Q(t_0,\om)g_0)\| \leq r_0(\om),\quad \text{a.s}.
	\eeq
	where
	\beq \bl{srd20}
	r_0 (\om) = \sqrt{1 + \frac{| \gw |}{\text{min}\{c_1, 1\}} \int_{-\infty}^{-1} e^{\sigma + \sigma s} \left[\left(2 c_2 + \frac{1}{32}c_1^2\right)e^{- 2 \ve \om (s)} + 2 (c_1 a)^4 e^{- 4 \ve \om (s)} \right] ds}
	\eeq
	in which both integrals 
	$$
	\int_{-\infty}^{-1} e^{\sigma + \sigma s} \left(2 c_2 + \frac{1}{32}c_1^2\right)e^{- 2 \ve \om (s)} \, ds \quad \text{and} \quad  \int_{-\infty}^{-1} 2 e^{\sigma + \sigma s} (c_1 a)^4 e^{- 4 \ve \om (s)}\, ds
	$$ 
	are convergent due to the asymptotically sublinear growth property \eqref{srd7}. 
	
	Therefore, the weak solution $G(t,\om;t_0,Q(t_0,\om)g_0)$ of the problem \eqref{sGq} uniquely exists on $[t_0, -1]$. The proof is completed.
\end{proof}

\begin{lemma} \label{theo2}
	There exists a random variable $R_0(\om)>0$ depending only on the parameters such that, for any given random variable $\rho(\om)>0$, the following statement holds. For any $t_0 \leq \tau(\rho,\om)$ specified in Lemma \ref{theo1} and any initial data $g_0=(u_0, v_0, z_0) \in H$ with $\|g_0\| \leq \rho(\om)$, the weak solution $G(t,\om;t_0, Q(t_0,\om)g_0)$ of the initial value problem \eqref{sGq} with $G(t_0,\om)=Q(t_0,\om)g_0$ uniquely exists on $[t_0,\infty)$ and satisfies
	\beq \bl{srd21}
	\|G(0,\om; t_0, Q(t_0,\om)g_0)\|^2 + \int_{-1}^{0} \|\nb G(s,\om;t_0,Q(t_0,\om)g_0)\|^2\, ds \leq R_0^2(\om), \;\;  \om \in \mathfrak{Q}.
	\eeq
\end{lemma}
\begin{proof}
	Based on Lemma \ref{theo1} and the local extension of the solutions of the problem \eqref{sGq} from the time $t_1 = - 1$ forward, we can integrate the inequality \eqref{srd13} over $[-1, t]$ to get
	\beq \bl{srd22}
	\begin{split}
		&c_1 \|U(t)\|^2 + \|V(t)\|^2 +\|Z(t)\|^2 - (c_1 \|U(-1)\|^2 + \|V(-1)\|^2 +\|Z(-1)\|^2) \\[3pt]
		&\, + 2d \int_{-1}^{t}\left(c_1 \|\nb U(s)\|^2 + \|\nb V(s)\|^2 + \|\nb Z(s)\|^2\right)\, ds \\
		&\, + \sigma \int_{-1}^{t}\left(c_1 \|U(s)\|^2 + \|V(s)\|^2 + \|Z(s)\|^2\right)\, ds  \\
		\leq &\, | \gw | \int_{-1}^{t} \left[\left(2 c_2 + \frac{1}{32}c_1^2\right)Q(s,\om)^2  + 2 (c_1 a)^4 Q(s,\om)^4 \right] ds, \quad t > -1.
	\end{split}
	\eeq
	Then 
	\beq \bl{srd23}
	\begin{split}
		& \|G(t,\om;t_0,Q(t_0,\om)g_0)\|^2 + 2d \int_{-1}^{t} \|\nb G(s,\om;t_0,Q(t_0,\om)g_0)\|^2\, ds\\[3pt]
		\leq &\, \frac{\text{max}\{c_1, 1\}}{\text{min}\{c_1, 1\}} \|G(-1,\om;t_0,Q(t_0,\om)g_0)\|^2 \\
		&\, + \frac{| \gw |}{\text{min}\{c_1, 1\}} \int_{-1}^{t} \left[\left(2 c_2 + \frac{1}{32}c_1^2\right)Q(s,\om)^2  + 2 (c_1 a)^4 Q(s,\om)^4 \right] ds.
	\end{split}
	\eeq
	The inequality \eqref{srd23} together with Lemma \ref{theo1} shows that for $\om \in \mathfrak{Q}$ and any $T> -1$, the weak solution $G(t,\om;t_0, Q(t_0,\om)g_0) \in C[t_0, T; H) \cap L^2(t_0, T; E)$ uniquely exists for $t \in [-1, T]$ and will not blow up. In particular, let $t=0$ in \eqref{srd23} and we obtain 
	\beq \bl{srd24}
	\|G(0,\om;t_0,Q(t_0,\om)g_0)\|^2 + \int_{-1}^{0} \|\nb G(s,\om;t_0,Q(t_0,\om)g_0)\|^2\, ds \leq R_0^2(\om), 
	\eeq
	where
	\beq \bl{srd25}
	\begin{split}
		&R_0^2(\om) = \frac{1}{\text{min}\{1, 2d\} \text{min}\{c_1, 1\}}\\[3pt]
		&\quad \times \left\{\text{max}\{c_1, 1\} |r_0(\om)|^2 + | \gw |\int_{-1}^{0} \left[\left(2 c_2 + \frac{1}{32}c_1^2\right)Q(s,\om)^2  + 2 (c_1 a)^4 Q(s,\om)^4 \right] ds \right\}
	\end{split}	
	\eeq
	where $r_0(\om)$ is defined in \eqref{srd20}. Note that $r_0(\om)$ and $R_0(\om)$ are both independent of random variables $\rho(\om)$.
\end{proof}

\begin{remark}
	We can certainly merge the above two lemmas into one which gives rise to the bounded estimate \eqref{srd24}. Here we split the time interval $[t_0, 0]$ to $[t_0, -1] \cup [-1, 0]$ in order to facilitate the argument in the proof of the pullback asymptotic compactness of the associated random dynamical system later in Section 3.
\end{remark}

\subsection{\textbf{Hindmarsh-Rose Cocycle and Absorbing Property}}

Now define a concept of stochastic semiflow, which is related to the concept of cocycle in the theory of random dynamical systems.

\begin{definition} \bl{SSM}
Let $(\mathfrak{Q}, \mathcal{F}, P, \{\theta_t\}_{t\in \mathbb{R}})$ be a metric dynamical system. A family of mappings $S(t,\tau,\om): X \to X$ for $t \geq \tau \in \mathbb{R}$ and $\om \in \mathfrak{Q}$ is called a \emph{stochastic semiflow} on a Banach space $X$, if it satisfies the properties:

	(i)\, $S(t,s, \om) S(s,\tau,\om) = S(t,\tau,\om)$, for all $\tau \leq s \leq t$ and $\om \in \mathfrak{Q}$.

	(ii) $S(t, \tau, \om) = S(t-\tau, 0, \theta_\tau \om)$, for all $\tau \leq t$ and $\om\in \mathfrak{Q}$.

	(iii) The mapping $S(t,\tau,\om) x$ is measurable in $(t,\tau,\om)$ and continuous in $x \in X$.
\end{definition}

Here in the setting of the stochastic evolutionary equation \eqref{sGq} formulated from the stochastic Hindmarsh-Rose equations \eqref{suq}-\eqref{inc1}, we define $S(t,\tau,\om): H \to H$ for $t \geq \tau \in \mathbb{R}$ and $\om \in \mathfrak{Q}$ by
\beq \bl{ssm}
	S(t, \tau, \om) \, g_0 = \frac{1}{Q(t, \om)}\, G(t,\,  \om; \, \tau, \, G_0) = \begin{pmatrix}
	u \\
	v \\
	z
	\end{pmatrix} 
	(t,\,\om; \, \tau, g_0)
\eeq
and then define a mapping $\Phi: \mathbb{R}^+ \times \mathfrak{Q} \times H \to H$, where $\mathbb{R}^+ = [0, \infty)$, to be
\beq \bl{cyl}
	\Phi (t - \tau, \, \theta_\tau \om, \, g_0) = S(t, \tau, \om)\, g_0
\eeq
which is equivalent to
\beq \bl{phi}
	\Phi (t,\, \om, \, g_0) = S(t, 0,  \om) g_0 = \frac{1}{Q(t,\om)}\, G(t, \,\om; \, 0, \, G_0).
\eeq
The following lemma shows that this mapping $\Phi$ is a cocycle on the Hilbert space $H$ over the canonical metric dynamical system $(\mathfrak{Q}, \mathcal{F}, P, \{{\theta_t}\}_{t \in \mathbb{R}})$ specified in \eqref{scrd5} and \eqref{srd6}. Therefore, the following \emph{pullback identity} is validated:
\beq \bl{vpt}
	\Phi (t, \theta_{-t} \om, g_0) = S(0, - t, \om) g_0 = \frac{1}{Q(0,\om)} G (0, \, \om; \, -t, \, G_0) = g (0, \, \om; \, -t, \, g_0)
\eeq
for any $t \geq 0$ and $\om \in \mathfrak{Q}$. We shall call this mapping $\Phi$ defined by \eqref{cyl} the \emph{Hindmarsh-Rose cocycle}, which is a random dynamical system on the Hilbert space $H$. We shall call $\{\Phi (t, \theta_{-t} \om, g_0): t \geq 0\}$ a \emph{pullback quasi-trajectory} with the initial state $g_0$ for the Hindmarsh-Rose cocycle.

\begin{remark}
Here the pullback quasi-trajectory $\{\Phi (t, \ctw, g_0), t \geq 0\}$ is not a single trajectory but the set of all the points at time $t = 0$ of the bunch of trajectories started from the same initial state $g_0$ but at different pullback initial time $-t$. 
\end{remark}

\begin{lemma} \bl{SPhi}
     The mapping $\Phi:  \mathbb{R}^+ \times \mathfrak{Q} \times H \to H$ defined by \eqref{ssm} and \eqref{cyl} is a cocycle on the space $H$ over the canonical metric dynamical system $(\mathfrak{Q}, \mathcal{F}, P, \{{\theta_t}\}_{t \in \mathbb{R}})$. Moreover, the one-parameter operators 
\beq \bl{Pit}
     (\Pi_t g)(\om) = \Phi (t, \theta_{-t} \om, g (\theta_{-t} \om)), \quad t \geq 0,
\eeq
where $\{g(\om): \om \in \mathfrak{Q}\}$ can be any $H$-valued random set on the probability space $(\mathfrak{Q}, \mathcal{F}, P)$, turns out to be a semigroup of operators on the $H$-valued random sets.
\end{lemma}

\begin{proof}
First we check the cocycle property of the mapping $\Phi$,
\beq \bl{ccy}
	\Phi (t + s, \om, g_0) = \Phi (t, \theta_s \, \om, \Phi (s, \om, g_0)), \quad  t \geq 0, \, s \geq 0, \, \om \in \mathfrak{Q},
\eeq
is satisfied by this mapping $\Phi$. Since we have \eqref{phi},
$$
	\Phi (t+s, \om, g_0) = \frac{1}{Q(t + s, \om)} G(t+s, \om; \, 0, \, G_0) 
$$
and, on the other hand,
\begin{align*} 
	&\Phi (t, \theta_s \, \om, \Phi (s, \om, g_0)) = \frac{1}{Q(t, \om)} G (t, \zsw; \, 0, \, G(s, \om; 0, G_0))  \quad (\tup{by}\, \eqref{phi}) \\
	= &\, g(t, \zsw; \, 0, \, g(s, \om; 0, g_0)) =  S(t, 0; \zsw)\, g(s, \om; 0, g_0)  \quad (\tup{by}\, \eqref{ssm}) \\[4pt]
	= &\, S(t, 0; \zsw)\, S(s, 0; \om) g_0  = S (t+s-s, 0; \zsw)\, S(s, 0; \om) g_0  \\[4pt]
	= &\, S(t+s, s; \om)\, S(s, 0; \om) g_0  \qquad (\tup{by the 2nd condition of Definition \ref{SSM}}) \\
	= &\, S(t+s, 0; \om)\, g_0 = \frac{1}{Q(t+s, \om)} G(t+s, \om; 0, G_0).  
\end{align*} 
Therefore, the cocycle property \eqref{ccy} of the mapping $\Phi$ is valid by comparison of the above two equalities. 	

The second claim that $\{\Pi_t\}_{t \geq 0}$ is a semigroup can be shown as follows,
\beq \bl{stg}
	\begin{split}
	(\Pi_t \, [\Pi_\gs \, g])(\om) &= \Phi (t, \, \theta_{-t} \om, \, [\Pi_\gs \, g] (\theta_{-t} \om)) \\[2pt]
	& = \Phi (t, \, \theta_{-t} \om, \, \Phi (\gs, \theta_{-\gs} (\theta_{-t} \om), g (\theta_{-\gs} (\theta_{-t} \om))) \\[2pt]
	& = \Phi (t, \, \theta_{-t} \om, \, \Phi (\gs, \theta_{-(t +\gs)} \om, g (\theta_{-(t+\gs)} \om))) \\[2pt]
	& = \Phi (t, \, \theta_{-(t+\gs)} \theta_{\gs} \om, \, \Phi (\gs, \theta_{-(t +\gs)} \om, g (\theta_{-(t+\gs)} \om))) \\[2pt]
	& = \Phi (t, \, \theta_{\gs} \theta_{-(t+\gs)} \om, \, \Phi (\gs, \theta_{-(t +\gs)} \om, g (\theta_{-(t+\gs)} \om))) \\[2pt]
	& = \Phi (t + \gs, \, \theta_{-(t +\gs)} \om, \, g (\theta_{-(t+\gs)} \om)) = (\Pi_{t + \gs} \, g)(\om), \;\; t, \gs \geq 0.
	\end{split}
\eeq
where the final equality follows from the cocycle property of $\Phi$ already proved.
\end{proof}

\begin{remark}
Apparently when the stochastic PDEs \eqref{suq}-\eqref{szq} are converted to the random PDEs \eqref{sUq}-\eqref{sZq} by the exponential multiplication \eqref{etrans}, we see the coefficients are time-depending random variables instead of constants,  which means the system \eqref{sGq} is nonautonomous in time. The justification for the corresponding stochastic semiflow \eqref{ssm} to be well-defined and satisfy the stationary property in Definition \ref{SSM} is due to the stationary property possessed by the underlying Wiender process $\{W(t)\}_{t \in \mathbb{R}}$, which is characterized by the stationary increment $W(t) - W(s)$ of the Gaussian distribution with mean zero and variance $t - s$, in the problem setting.
\end{remark}

\begin{theorem} \bl{theo3}
	There exists a pullback absorbing set in the space $H$ with respect to the tempered universe $\mathscr{D}_H$ for the Hindmarsh-Rose cocycle $\Phi$, which is the bounded random ball
	\beq \bl{srd26} 
	B_0(\om) = B_H(0, R_0(\om)) = \{\xi \in H: \|\xi\| \leq R_0(\om)\}
	\eeq
	where $R_0(\om)$ is given in \eqref{srd25}.
\end{theorem}

\begin{proof}
For any bounded random ball $B(\om) = B_H(0, \rho(\om)) \in \mathscr{D}_H$ and any $g_0 \in B(\theta_{-t}\om)$, by Definition \ref{sc3} we have
\beq \bl{srd27}
\lim_{t \to -\infty} e^{-\beta t} \rho(\theta_{-t} \om) = 0,\quad \text{for any} \; \beta >0.
\eeq
From \eqref{srd16}, for $-t \leq -1$ we have
\begin{equation*}
	\begin{split}
	 \sup_{g_0 \in B(\theta_{-t}\om)} &\; \|G(-1, \om; -t, Q(-t,\om)g_0)\|^2 \leq \frac{\|Q(-t,\om)\|^2 \text{max}\{c_1, 1\}}{\text{min}\{c_1, 1\}} e^{ \sigma (1 - t)} \|g_0\|^2 \\[3pt]
	&\; + \frac{| \gw|}{\text{min}\{c_1, 1\}} \int_{-\infty}^{-1} e^{\sigma(1+s)} \left[\left(2 c_2 + \frac{1}{32}c_1^2\right)Q(t,s)^2 + 2 (c_1 a)^4 Q(t,s)^4 \right] ds\\[3pt]
	\leq &\;\frac{e^{\sigma} \text{max}\{c_1, 1\}}{\text{min}\{c_1, 1\}} e^{-2\ve \om(-t) - \sigma t} \rho^2(\theta_{-t}\om)\quad\quad\quad  \left(\text{since}\; g_0 \in B(\theta_{-t}\om)\right)  \\
	&\; + \frac{| \gw |}{\text{min}\{c_1, 1\}} \int_{-\infty}^{-1} e^{\sigma(1+s)} \left[\left(2 c_2 + \frac{1}{32}c_1^2\right) e^{-2 \ve \om(s)} + 2 (c_1 a)^4 e^{-4 \ve \om(s)} \right] ds. \\
	\leq&\; \frac{e^{\sigma} \text{max}\{c_1, 1\}}{\text{min}\{c_1, 1\}} \, \text{exp} \left[-\frac{\sigma t}{2} \left(1 - \frac{4 \ve}{\sigma} \left(\frac{\om(-t)}{-t}\right)\right)\right] e^{-\frac{\sigma t}{2}} \rho^2(\theta_{-t}\om) \\
	&\; + \frac{| \gw|}{\text{min}\{c_1, 1\}} \int_{-\infty}^{-1} e^{\sigma(1+s)} \left[\left(2 c_2 + \frac{1}{32}c_1^2\right) e^{-2 \ve \om(s)} + 2 (c_1 a)^4 e^{-4 \ve \om(s)} \right] ds.
	\end{split}
\end{equation*}
From \eqref{srd7}, we have
$$
\lim_{t \to \infty} \text{exp} \left[-\frac{\sigma t}{2} \left(1 - \frac{4 \ve}{\sigma} \left(\frac{\om(-t)}{-t}\right)\right)\right] = 0, \quad \om \in \mathfrak{Q}. 
$$
Since $B(\om) = B_H(0, \rho(\om)) \in \mathscr{D}_H$, the radius $\rho(\theta_{-t}\om)$ is a tempered random variable, so that 
$$
\lim_{t \to \infty} e^{-\frac{\sigma t}{2}} \rho^2(\theta_{-t}\om) = \lim_{t \to \infty} |e^{-\frac{\sigma t}{4}} \rho(\theta_{-t}\om)|^2 = 0.
$$
Therefore, there exists a finite random variable $T_B(\om) > 1$ such that for all $t \geq T_B(\om)$ we have 
$$
\frac{e^{\sigma} \text{max}\{c_1, 1\}}{\text{min}\{c_1, 1\}} \text{exp} \left[-\frac{\sigma t}{2} \left(1 - \frac{4 \ve}{\sigma} \left(\frac{\om(-t)}{-t}\right)\right)\right] \leq 1 \;\; \text{and} \;\; e^{-\frac{\sigma t}{2}} \rho^2(\theta_{-t}\om) \leq 1,\;\, \om \in \mathfrak{Q}.
$$
Then 
\beq \bl{Gr0}
	\sup_{g_0 \in B(\theta_{-t}\om)} \|G(-1, \theta_{-t}\om; -t, Q(-t,\om)g_0)\| \leq r_0(\om),\quad \text{for}\,\; t\geq T_B(\om),\;\, \om \in \mathfrak{Q},
\eeq
where $r_0(\om)$ is given in \eqref{srd20}.

Finally, put together \eqref{srd24}, \eqref{srd25} and \eqref{Gr0}. We end up with 
\beq \bl{Ab}
	\sup_{g_0 \in B(\theta_{-t}\om)} \|\Phi(t, \theta_{-t} \om, g_0)\| = \sup_{g_0 \in B(\theta_{-t}\om)} \|G(0, \theta_{-t}\om; -t, Q(-t,\om)g_0)\| \leq R_0(\om), 
\eeq
for $t \geq T_B(\om)$ a.s. Hence, the random set in \eqref{srd26} is a pullback absorbing set for the Hindmarsh-Rose cocycle $\Phi$. The proof is completed.
\end{proof}

\section{\textbf{The Existence of Random Attractor}}

In this section, we shall prove that this Hindmarsh-Rose cocycle is pullback asymptotically compact on $H$ through the following two lemmas. Then the main result on the existence of a random attractor for the Hindmarsh-Rose cocycle is established. 

\begin{lemma} \bl{theo4} 
	Assume that for any random variable $R(\om) > 0$ and any given $\tau < -2$, there exists a random variable $M(R, \om) > 0$ such that the following statement is valid\textup{:} If there is a time $t^* \in [-2, -1]$ such that $G(t^*, \om;\,\tau, Q(\tau,\om)g_0) \in E$ for any $g_0 \in H$ which satisfies
	$$
	\|G(t^*,\, \om; \,\tau, \,Q(\tau,\om)g_0)\|_E \leq R(\om), 
	$$
	then it holds that 
	\beq \bl{srd28}
		\|G(0,\, \om; \, \tau,\, Q(\tau,\om)g_0)\|_E \leq M(R, \om).
	\eeq
\end{lemma}

\begin{proof}
	Denote the solution of \eqref{sGq} by $G(t,\om;\tau,Q(\tau,\om)g_0) = (U(t), V(t), Z(t))$. Take the $L^2$ inner-product $\inpt{\eqref{sUq}, -\gd U(t)}$ to obtain
	\begin{equation*}
	\begin{split}
	&\frac{1}{2} \frac{d}{dt} \|\nb U\|^2 + d_1 \|\gd U\|^2 \\[3pt]
	= & \int_\gw \left(- \frac{a}{Q(t,\om)} U^2 \gd U - \frac{b}{Q(t,\om)^2} U^2 |\nb U|^2 -V \gd U + Z \gd U - J Q(t,\om) \gd U \right)\, dx \\
	\leq & \int_\gw\left(\frac{2 a^2}{d_1 Q(t,\om)^2} U^4 + \frac{d_1}{8} |\gd U|^2 + \frac{2}{d_1}V^2 + \frac{d_1}{8} |\gd U|^2 \right) ds \\
	 &\, +  \int_\gw  \left(\frac{2}{d_1}Z^2 + \frac{d_1}{8} |\gd U|^2 + \frac{2 J^2 Q(t,\om)^2}{d_1} + \frac{d_1}{8} |\gd U|^2 \right) dx - \int_\gw \frac{b}{Q(t,\om)^2} U^2 |\nb U|^2 \, dx.
	\end{split}
	\end{equation*}
	It follows that
	\beq \bl{srd29}
	\begin{split}
		& \frac{d}{dt} \|\nb U\|^2 + d_1 \|\gd U\|^2 + \frac{2b}{Q(t,\om)^2} \|U \nb U\|^2\\[3pt]
		\leq &\, \frac{4 a^2}{d_1 Q(t,\om)^2} \|U\|^4_{L^4} + \frac{4}{d_1} \|V\|^2 + \frac{4}{d_1} \|Z\|^2 + \frac{4 J^2 Q(t,\om)^2}{d_1} |\gw |,\quad t > \tau.
	\end{split}
	\eeq
	Take the $L^2$ inner-product $\inpt{\eqref{sVq}, -\gd V(t)}$, we get
	\begin{equation*}
	\begin{split}
		 & \frac{1}{2} \frac{d}{dt} \|\nb V\|^2 + d_2 \|\gd V\|^2 + \|\nb V\|^2 \\[3pt]
		= & \int_\gw \left(-\alpha Q(t,\om) \gd V + \frac{\beta}{Q(t,\om)} U^2 \gd V + V \gd V\right) dx\\
		\leq & \int_\gw \left(\frac{\alpha^2 Q(t,\om)^2}{d_2} + \frac{d_2}{4} |\gd V|^2 + \frac{\beta^2}{d_2 Q(t,\om)^2} U^4 + \frac{d_2}{4} |\gd V|^2 - |\nb V|^2\right) dx.
	\end{split}
	\end{equation*}
	Then
	\beq \bl{srd30}
		\frac{d}{dt} \|\nb V\|^2 + d_2 \|\gd V\|^2 + 2 \|\nb V\|^2 \leq  \frac{2 \alpha^2 Q(t,\om)^2}{d_2} |\gw | + \frac{2 \beta^2}{d_2 Q(t,\om)^2} \|U\|^4_{L^4},\quad t >\tau.
	\eeq
	Take the $L^2$ inner-product $\inpt{\eqref{sZq}, -\gd Z(t)}$, we get
	\begin{equation*}
	\begin{split}
	& \frac{1}{2} \frac{d}{dt} \|\nb Z\|^2 + d_3 \|\gd Z\|^2 \\[3pt]
	= & \int_\gw (qcQ(t,\om) \gd Z - q U \gd Z + rZ \gd Z)\, dx\\
	\leq &\, \int_\gw \left(\frac{q^2 c^2 Q(t,\om)^2}{d_3} + \frac{d_3}{4} |\gd Z|^2 + \frac{q^2}{d_3}U^2 + \frac{d_3}{4} |\gd Z|^2 - r|\nb Z|^2 \right) dx.
	\end{split}
	\end{equation*}
	It implies
	\beq \bl{srd31}
		\frac{d}{dt} \|\nb Z\|^2 + d_3 \|\gd Z\|^2 + 2 r \|\nb Z\|^2 \leq \frac{2 q^2 c^2 Q(t,\om)^2}{d_3} |\gw | +  \frac{2 q^2}{d_3}\|U\|^2,\quad t> \tau.
	\eeq
	Sum up the above estimates \eqref{srd29}, \eqref{srd30} and \eqref{srd31}. Then we obtain
	
	\beq \bl{srd32}
	\begin{split}
		& \frac{d}{dt} (\|\nb U\|^2 + \|\nb V\|^2 + \|\nb Z\|^2) + d_1 \|\gd U\|^2 + d_2 \|\gd V\|^2 + d_3 \|\gd Z\|^2 \\[3pt]
		&\; + \frac{2b}{Q(t,\om)^2} \|U \nb U\|^2 + 2 \|\nb V\|^2 + r \|\nb Z\|^2 \\
		 \leq &\; \frac{2 q^2}{d_3}\|U\|^2 +\frac{4}{d_1} \|V\|^2 + \frac{4}{d_1} \|Z\|^2 + \frac{1}{Q(t,\om)^2} \left(\frac{4 a^2}{d_1} + \frac{2 \beta^2}{d_2} \right) \|U\|^4_{L^4}\\
		&\; +  Q(t,\om)^2 \left(\frac{4 J^2}{d_1} + \frac{2 \alpha^2}{d_2} + \frac{2 q^2 c^2}{d_3} \right) |\gw |.
	\end{split}
	\eeq
	Since $H^1(\gw) \hookrightarrow L^4 (\gw)$, there is a positive constant $\eta > 0$ associated with the Sobolev imbedding inequality such that
	$$
	\|U\|^4_{L^4} \leq \eta (\|U\|^2 + \|\nb U\|^2)^2 \leq 2 \eta (\|U\|^4 +\|\nb U \|^4).
	$$
	For any $t \in [t^*, 0] \subset [\tau, 0]$, the inequality \eqref{srd16} implies that
	\beq \bl{srd33}
	\begin{split}
		&\; \|G(t,\om;\tau, Q(\tau,\om)g_0)\|^2 \\[3pt]
		\leq &\,  \frac{\text{max}\{c_1, 1\}}{\text{min}\{c_1, 1\}} \|G(t^*,\om;\tau, Q(\tau,\om)g_0)\|^2    \\
		&\, + \frac{| \gw |}{\text{min}\{c_1, 1\}} \int_{-\infty}^{0} e^{\sigma s} \left[\left(2 c_2 + \frac{1}{32}c_1^2\right) Q(s,\om)^2  + 2 (c_1 a)^4 Q(s,\om)^4 \right] ds, 
	\end{split}
	\eeq
where the improper integralin \eqref{srd33} is convergent due to \eqref{srd7} and $\sigma = \frac{1}{2} \min \{1, r\} > 0$, as given after \eqref{srd13}. Denote by
	\begin{equation*}
	\begin{split}
		P_0(R,\om) &\; = \frac{\text{max}\{c_1, 1\}}{\text{min}\{c_1, 1\}} R^2(\om) \\[3pt] 
		&\; + \frac{| \gw |}{\text{min}\{c_1, 1\}} \int_{-\infty}^{0} e^{\sigma s} \left[\left(2 c_2 + \frac{1}{32}c_1^2\right) Q(s,\om)^2  + 2 (c_1 a)^4 Q(s,\om)^4 \right] ds.
	\end{split}
	\end{equation*}
	Then from \eqref{srd32} we obtain
	\beq \bl{srd34}
		\begin{split}
			&\; \frac{d}{dt} \|\nb G\|^2 + d_1 \|\gd U\|^2 + d_2 \|\gd V\|^2 + d_3 \|\gd Z\|^2 \\[3pt]
			&\; + \frac{2b}{Q(t,\om)^2} \|U \nb U\|^2 + 2 \|\nb V\|^2 + r \|\nb Z\|^2 \\
			\leq &\; \text{max}\left\{\frac{2 q^2}{d_3}, \frac{4}{d_1}\right\} P_0(R,\om) + \frac{2 \eta}{Q(t,\om)^2} \left(\frac{4 a^2}{d_1} + \frac{2 \beta^2}{d_2} \right) P^2_0(R,\om)\\
	&\; + \frac{2 \eta}{Q(t,\om)^2} \left(\frac{4 a^2}{d_1} + \frac{2 \beta^2}{d_2} \right) \|\nb U\|^4+  Q(t,\om)^2 \left(\frac{4 J^2}{d_1} + \frac{2 \alpha^2}{d_2} + \frac{2 q^2 c^2}{d_3} \right) |\gw |.
		\end{split}
	\eeq
	Here we can apply the uniform Gronwall inequality to the following inequality
	\beq \bl{srd35}
	\begin{split}
		\frac{d}{dt} \|\nb& G(t)\|^2  \leq  \frac{2 \eta}{Q(t,\om)^2} \left[\frac{4 a^2}{d_1} + \frac{2 \beta^2}{d_2} \right] \|\nb G(t)\|^2 \|\nb G(t)\|^2 + \text{max}\left\{\frac{2 q^2}{d_3}, \frac{4}{d_1}\right\} P_0(R,\om)\\
   & + \frac{2 \eta}{Q(t,\om)^2} \left[\frac{4 a^2}{d_1} + \frac{2 \beta^2}{d_2} \right] P^2_0(R,\om) + Q(t,\om)^2 \left(\frac{4 J^2}{d_1} + \frac{2 \alpha^2}{d_2} + \frac{2 q^2 c^2}{d_3} \right) |\gw |
	\end{split}
	\eeq
	for $t \geq t^*$, which is written in the form
	\beq \bl{srd36}
	\frac{d \xi}{dt} \leq p\, \xi + h,
	\eeq
	where 
	\begin{equation*}
	\begin{split}
		&\; \xi (t) =  \|\nb G (t)\|^2, \\[3pt]
		&\; p(t) = \frac{2 \eta}{Q(t,\om)^2} \left(\frac{4 a^2}{d_1} + \frac{2 \beta^2}{d_2} \right) \|\nb G(t)\|^2,\\
		&\; h(t) = \text{max}\left\{\frac{2 q^2}{d_3}, \frac{4}{d_1}\right\} P_0(R,\om) + \frac{2 \eta}{Q(t,\om)^2} \left(\frac{4 a^2}{d_1} + \frac{2 \beta^2}{d_2} \right) P^2_0(R,\om) \\
		&\quad \quad \quad + Q(t,\om)^2 \left(\frac{4 J^2}{d_1} + \frac{2 \alpha^2}{d_2} + \frac{2 q^2 c^2}{d_3} \right) |\gw |.
	\end{split}
	\end{equation*}
By integration of the inequality \eqref{srd13} over $[t, t+1]$ for $t \in [t^*, -1]$, we can deduce that
	\begin{equation*}
	\begin{split}
		 \int_{t}^{t+1} 2d (c_1 \|\nb U(s)\|^2 &+ \|\nb V(s)\|^2 + \|\nb Z(s)\|^2)\,ds \leq c_1 \|U(t)\|^2 + \|V(t)\|^2 + \|Z(t)\|^2\\[3pt]
		&\, + \int_{t}^{t+1} \left[\left(2 c_2 + \frac{1}{32}c_1^2\right)Q(s,\om)^2 | \gw | + 2 (c_1 a)^4 Q(s,\om)^4 |\gw | \right]\, ds.
	\end{split}
	\end{equation*}
	Since $Q(t,\om) = e^{-\ve \om(t)}$, the above inequality implies that, for $t \in [t^*, -1]$,
	\begin{equation} \bl{srd37}
	\begin{split}
		 &\int_{t}^{t+1} \xi (s)\, ds \leq \frac{\max \{c_1,1\}}{2 d \;\text{min}\{c_1, 1\}} P_0(R,\om)    \\[3pt]
		 + &\, \frac{1}{2 d \;\text{min}\{c_1, 1\}} \int_{t}^{t+1} \left[\left(2 c_2 + \frac{1}{32}c_1^2\right)e^{-2 \ve \om(s)} |\gw | + (c_1 a)^4 e^{-4 \ve \om(s)} |\gw | \right] ds.
	\end{split}
	\end{equation}
Here $e^{-\ve \om(t)}$ is continuous function on $[-2, 0]$, so that there is a bound 
$$
	|Q(t, \om)| = e^{-\ve \om(t)} \leq e^{\ve |\om(t)|} \leq C(\om) =  \exp \left(\ve \sup_{t \in [-2, 0]} |\om (t)| \right), \quad t \in [-2, 0].
$$
Then \eqref{srd37} implies that for any $\tau < -2$ and $t \in [t^*, -1]$, 
	\beq \bl{srd38}
	\begin{split}
	\int_{t}^{t+1} \xi (s)\, ds &\;\leq N_1 (R, \om),
	\end{split}
	\eeq
	where
	\begin{equation*}
	\begin{split}
	N_1 (R, \om) &\; =  \frac{1}{2 d \;\text{min}\{c_1, 1\}} \\[3pt]
	&\; \times \left\{\text{max}\{c_1,1\} P_0(R,\om) + \left(2 c_2 + \frac{1}{32}c_1^2\right) C^2(\om) |\gw | + (c_1 a)^4 C^4(\om) |\gw | \right\}.
	\end{split}
	\end{equation*}
Next we have
	\beq \bl{srd39}
	\begin{split}
		\int_{t}^{t+1} p(s)\, ds &\; \leq \int_{t}^{t+1} 2 \eta  \left(\frac{4 a^2}{d_1} + \frac{2 \beta^2}{d_2} \right) \frac{1}{Q(s,\om)^2} \|\nb G(s)\|^2\; ds\\[3pt]
		&\; \leq 2 \eta \,C^2 (\om) \left(\frac{4 a^2}{d_1} + \frac{2 \beta^2}{d_2} \right) \int_{t}^{t+1} \|\nb G(s)\|^2\; ds \leq N_2 (R, \om),
	\end{split}
	\eeq
	where
	$$
	N_2 (R, \om) = 2 \eta \,C^2(\om) \left(\frac{4 a^2}{d_1} + \frac{2 \beta^2}{d_2} \right) N_1(R, \om).
        $$
Moreover, for any $\tau < -2$ and $t \in [t^*, -1]$, we obtain
	\beq \bl{srd40}
	\begin{split}
		\int_{t}^{t+1} h(s)\, ds &\; \leq \int_{t}^{t+1} \left[\text{max}\left\{\frac{2 q^2}{d_3}, \frac{4}{d_1}\right\} P_0(R,\om) + \frac{2 \eta}{Q(s,\om)^2} \left(\frac{4 a^2}{d_1} + \frac{2 \beta^2}{d_2} \right) P^2_0(R,\om) \right. \\[2pt]
		&\quad \left. + Q(s,\om)^2 \left(\frac{4 J^2}{d_1} + \frac{2 \alpha^2}{d_2} + \frac{2 q^2 c^2}{d_3} \right) |\gw | \right]\; ds\\[2pt]
		&\; \leq \text{max}\left\{\frac{2 q^2}{d_3}, \frac{4}{d_1}\right\}P_0(R,\om) + 2 \eta C^2(\om) \left(\frac{4 a^2}{d_1} + \frac{2 \beta^2}{d_2} \right) P^2_0(R,\om)\\[2pt]
		&\quad + C^2(\om) \left(\frac{4 J^2}{d_1} + \frac{2 \alpha^2}{d_2} + \frac{2 q^2 c^2}{d_3} \right) |\gw | = N_3 (R, \om).
	\end{split}
	\eeq
Now we have shown that, for any $\tau < -2$ and $t \in [t^*, -1]$,
	\beq \bl{srd41}
	\int_{t}^{t+1} \sigma(s)\, ds \leq N_1, \quad \int_{t}^{t+1} p(s)\, ds \leq N_2,  \quad  \int_{t}^{t+1} h(s)\, ds \leq N_3.
	\eeq
	Thus the uniform Gronwall inequality \cite[Lemma D.3]{SY} applied to \eqref{srd36} shows that
	\beq \bl{srd42}
	\xi (t) = \|\nb G(t)\|^2 \leq (N_1 + N_3) e^{N_2}, \quad \text{for all}\;\; t \in [t^* +1, 0].
	\eeq
Finally, the claim  \eqref{srd28} is proved:
	\begin{equation*}
	\begin{split}
		&\|G(0,\om;\tau,Q(\tau,\om)g_0)\|_E^2  = \|G(0,\om;\tau,Q(\tau,\om)g_0)\|^2 + \|\nb G(0,\om;\tau,Q(\tau,\om)g_0)\|^2 \\[3pt]
	        \leq &\, M(R, \om) = P_0(R,\om) + (N_1(R, \om) + N_3(R, \om)) e^{N_2(R,\, \om)}.
	\end{split}
	\end{equation*}
The proof is completed.
\end{proof}

\begin{lemma} \bl{theo5}
	For the Hindmarsh-Rose cocycle $\Phi$, there exists a random variable $M^*(\om) > 0$ with the property that for any given random variable $\rho(\om) > 0$ there is a finite time $T(\rho,\om) > 0$ such that if $g_0=(u_0, v_0, z_0) \in H$ with $\|g_0\| \leq \rho(\om)$, then $\Phi(t, \theta_{-t} \om,g_0) \in E$ and
	\beq \bl{srd44}
	\|\Phi(t, \theta_{-t},\om,g_0)\|_E \leq M^*(\om), \quad \text{for}\;\; t > T(\rho,\om).
	\eeq
\end{lemma}

\begin{proof}
	We have proved in Theorem \ref{theo3} the existence of a pullback absorbing set $B_0(\om) = B_H(0, R_0(\om))$ for the Hindmarsh-Rose cocycle $\Phi$ in $H$. Thus it suffices to show that the above statement \eqref{srd44} holds for $\rho(\om) = R_0(\om)$ given in \eqref{srd25}, namely, for $g_0 \in B_0(\om)$.
	
	From \eqref{srd16}, for any $g_0 \in B_0(\om)$, we obtain
	\beq \bl{srd45}
	\begin{split}
		&\|G(t,\om;\tau,Q(\tau,\om)g_0)\|^2 \leq \frac{\text{max}\{c_1, 1\}}{\text{min}\{c_1, 1\}} e^{- \sigma (t - \tau)} |Q(\tau,\om)|^2 R_0^2(\om)\\
		&\; + \frac{|\gw |}{\text{min}\{c_1, 1\}} \int_{- \infty}^t e^{- \sigma(t-s)} \left[\left(2 c_2 + \frac{1}{32}c_1^2\right)Q(s,\om)^2 + 2 (c_1 a)^4 Q(s,\om)^4 \right] ds.
	\end{split}
	\eeq
	Now we prove that there exists a time $T^*(R_0(\om)) < -2$ such that for any $\tau \leq T^*(R_0)$ one has 
	\beq \bl{srd46}
		\sup_{t \in [-2,0]} \; \sup_{g_0 \in B_0(\om)} \|G(t,\om;\tau,Q(\tau,\om)g_0)\| \leq R_1 (\om),
	\eeq
where $R_1(\om) > 0$ is a positive random variable given in \eqref{srd23c} later in this proof. 
	
Take $t=-2$ and recall that $Q(\tau, \om) = e^{-\ve\om (\tau)}$. The inequality \eqref{srd45} implies 
	\beq \bl{srd17a}
	\begin{split}
		&\|G(-2,\om;\tau,Q(\tau,\om)g_0)\|^2 \leq \frac{ \text{max}\{c_1, 1\}}{\text{min}\{c_1, 1\}} e^{2 \sigma - \sigma |\tau| - 2\ve\om (\tau)} R_0^2(\om)   \\
		&\; + \frac{|\gw |}{\text{min}\{c_1, 1\}} \int_{-\infty}^{-2} e^{2 \sigma + \sigma s} \left[\left(2 c_2 + \frac{1}{32}c_1^2\right)e^{- 2 \ve \om (s)} + 2 (c_1 a)^4 e^{- 4 \ve \om (s)} \right]\, ds.
	\end{split}
	\eeq
	Note that $\tau \leq T^*(R_0) < -2$ implies 
	$$
	e^{- \sigma |\tau| - 2\ve\om (\tau)} = \text{exp}\left(-\sigma |\tau|\left[1 + \frac{2\ve\om (\tau)}{\sigma |\tau|}\right]\right) = \text{exp}\left(-\sigma |\tau|\left[1 - \frac{2\ve\om (\tau)}{\sigma \tau}\right]\right). 
	$$
	By the asymptotically sublinear growth property \eqref{srd7}, for $\om \in \mathfrak{Q}$, there exist a time $T^*(R_0) \leq -2$ such that for any $\tau \leq T^*(R_0)$, which means $\tau$ is very negative, we have 
	\beq \bl{srd18a}
	1 - \frac{2\ve\om (\tau)}{\sigma \tau} \geq \frac{1}{2}\quad \text{and}\quad e^{\sigma (2 - \frac{1}{2}|\tau|)}\frac{ \text{max}\{c_1, 1\}}{\text{min}\{c_1, 1\}} R_0^2(\om) \leq 1.
	\eeq
	Then we get
	\begin{equation} \label{srd18b}
		\begin{split}
		&\; \|G(-2,\om;\tau,Q(\tau,\om)g_0)\|^2 \\[3pt]
		\leq &\; 1 + \frac{| \gw |}{\text{min}\{c_1, 1\}} \int_{-\infty}^{-2} e^{2 \sigma + \sigma s} \left[\left(2 c_2 + \frac{1}{32}c_1^2\right)e^{- 2 \ve \om (s)} + 2  (c_1 a)^4 e^{- 4 \ve \om (s)} \right] ds\\
		\leq &\; 1 + \frac{| \gw |}{\text{min}\{c_1, 1\}} \int_{-\infty}^{-1} e^{2 \sigma + \sigma s} \left[\left(2 c_2 + \frac{1}{32}c_1^2\right)e^{- 2 \ve \om (s)} + 2 (c_1 a)^4 e^{- 4 \ve \om (s)} \right] ds\\
		= &\; r_0^2(\om),
		\end{split}
		\end{equation}
	where $r_0(\om)$ is given in \eqref{srd20}.
	
	For $t \in [-2,0]$, integrate the inquality \eqref{srd13} over $[-2, t]$ to obtain
	\beq \bl{srd23a}
	\begin{split}
		&\, \|G(t,\om;\tau,Q(\tau,\om)g_0)\|^2 + 2d \int_{-2}^{t} \|\nb G(s,\om;\tau,Q(\tau,\om)g_0)\|^2\, ds\\[3pt]
		\leq &\, \frac{\text{max}\{c_1, 1\}}{\text{min}\{c_1, 1\}} \|G(-2,\om;\tau,Q(\tau,\om)g_0)\|^2 \\
		&\, +  \frac{| \gw|}{\text{min}\{c_1, 1\}} \int_{-2}^{t} \left[\left(2 c_2 + \frac{1}{32}c_1^2\right)Q(s,\om)^2  + 2 (c_1 a)^4 Q(s,\om)^4 \right] ds.
	\end{split}
	\eeq
The inequalities \eqref{srd18b} and \eqref{srd23a} imply that \eqref{srd46} is valid:
	\beq \bl{srd23b}
		 \|G(t,\om;\tau,Q(\tau,\om)g_0)\|^2 \leq R_1 (\om), \quad \text{for all}\;\; t \in [-2,0], \; g_0 \in B_0 (\om).
	\eeq
	where 
	\beq \bl{srd23c}
	\begin{split}
		R_1 (\om) &\, = \frac{\max \{c_1, 1\}}{\min \{c_1, 1\}} r_0^2 (\om)    \\[3pt]
		&\, + \frac{| \gw |}{\min \{c_1, 1\}} \int_{-2}^{0} \left[\left(2 c_2 + \frac{1}{32}c_1^2\right)Q(s,\om)^2  + 2 (c_1 a)^4 Q(s,\om)^4 \right] ds.
	\end{split}
	\eeq
	
Next for $t \geq -2$ and $\tau < T^* (R_0)$, we integrate \eqref{srd13} and by \eqref{srd46} to get
	\beq \bl{srd47}
	\begin{split}
		& \int_{t}^{t+1}	 \|\nb G(s,\om;\tau,Q(\tau,\om)g_0)\|^2\, ds \leq \frac{\max \{c_1, 1\}}{2d\, \min \{c_1, 1\}} \|G(t,\om;\tau,Q(\tau,\om)g_0)\|^2   \\[3pt]
	      + &\, \frac{| \gw |}{2d\, \min \{c_1, 1\}} \int_{t}^{t+1} \left[\left(2 c_2 + \frac{1}{32}c_1^2\right)Q(s,\om)^2 + 2 (c_1 a)^4 Q(s,\om)^4\ \right] ds \leq K(\om),
	\end{split}
	\eeq
where
	\begin{equation*}
	\begin{split}
	K(\om) &\, = \frac{1}{2d\, \min \{c_1, 1\}} \max \left\{c_1, \, 1, \, \left(2 c_2 + c_1^2\right) |\gw |, \, 2(c_1 a)^4 |\gw | \right\}   \\[3pt]
	 &\quad \times  \left\{R_1 (\om) + \int_{-2}^{0} \left[Q(s,\om)^2 + 2 Q(s,\om)^4\right] ds \right\}.
	\end{split}
	\end{equation*}
Take $t = -2$ and $\tau < T^*(R_0)$ in \eqref{srd47}. It implies that there is a time $t^* \in [-2,-1]$ such that 
	$$
	\|\nb G(t^*,\om; \tau, Q(\tau,\om)g_0)\|^2 \leq K(\om),
	$$
so that 
	\beq \bl{srd48}
		 \|G(t^*, \om; \, \tau, \, Q(\tau, \om)g_0)\|_E^2 \leq R_1(\om) + K(\om).
	\eeq
	
Finally, we combine Lemma \ref{theo4} and the bound estimate \eqref{srd48} to conclude that for all $t > |T^*(R_0(\om))|$ it holds that
	\beq \bl{srd49}
	\|\Phi (t, \theta_{-t} \om, g_0)\|_E = \|G (0, \om; -t, \, Q(-t,\om) g_0\|_E \leq M((R_1 + K)^{1/2}, \om )
	\eeq
where $M(R,\om)$ is specified in \eqref{srd28}. Thus the claim \eqref{srd44} of this lemma is proved for $\rho(\om) = R_0(\om)$ with 
	$$
	M^*(\om) = M((R_1 + K)^{1/2}, \om) \quad \text{and} \quad T(\rho,\om) =  |T^*(R_0(\om))|.
	$$
Consequently, \eqref{srd44} is also proved for any random variable $\rho (\om)$ as well, by the remark at the beginning of this proof. It completes the proof.
\end{proof}

We complete this paper to present the main result on the existence of a random attractor for the Hindmarsh-Rose random dynamical system $\Phi$ in the space $H$.
\begin{theorem} \bl{RAC}
	For any positive parameters $d_1, d_2, d_3, a, b, \alpha, \beta, q, r, J, \,\ve$ and $c \in \mathbb{R}$, there exists a random attractor $\mathcal{A}(\om)$ in the space $H = L^2 (\gw, \mathbb{R}^3)$ with respect to the universe $\mathscr{D}_H$ for the Hindmarsh-Rose cocycle $\Phi$ over the metric dynamical system $(\mathfrak{Q}, \mathcal{F}, P, \{{\theta_t}\}_{t \in \mathbb{R}})$. Moreover, the random attractor $\mathcal{A}(\om)$ is a bounded random set in the space $E$.
\end{theorem}

\begin{proof}
	In Lemma \ref{theo3}, we proved that there exists a pullback absorbing set $B_0(\om)$ in $H$ for the Hindmarsh-Rose cocycle $\Phi$. According to Definition \ref{sc6}, Lemma \ref{theo5} and the compact imbedding $E \hookrightarrow H$ show that the Hindmarsh-Rose cocycle $\Phi$ is pullback asymptotically compact on $H$ with respect to $\mathscr{D}_H$. Hence, by Theorem \ref{sc8}, there exists a random attractor in $H$ for this random dynamical system $\Phi$, which is given by
	\beq \bl{srd50}
	\mathcal{A} (\omega) = \bigcap_{\tau \geq 0} \; {\overline{\bigcup_{ t \geq \tau} \Phi (t, \theta_{-t} \omega, B_0 (\theta_{-t} \omega))}}.
	\eeq
Since $\mathcal{A}(\om)$ is an invariant set, Lemma \ref{theo5} implies that the random attractor $\mathcal{A}(\om)$ is also a bounded random set in $E$.
\end{proof}

\bibliographystyle{amsplain}

\end{document}